\documentclass[11pt,reqno]{amsart}
\usepackage[latin1]{inputenc}
\usepackage{amsmath}
\usepackage{url}
\usepackage{graphicx}
\usepackage{color}
\usepackage{amsfonts,amsmath,amssymb,amsthm}	
\usepackage{graphicx}			
\usepackage{hyperref}			
\usepackage{dsfont}
\usepackage{a4wide}
\usepackage{pstricks,pstricks-add,pst-math,pst-xkey}
\usepackage{multicol}

\numberwithin{equation}{section}

\def\R{\mathbb{R}}	

\def\N{\mathbb{N}}

\newtheorem{lemma}{Lemma}[section]
\newtheorem{proposition}{Proposition}[section]
\newtheorem{theorem}{Theorem}[section]
\newtheorem{corollary}{Corollary}
\newtheorem{definition}{Definition}[section]
\newtheorem{remark}{Remark}

\title[\textbf{ }]{Stabilization of two strongly coupled hyperbolic equations in exterior domains}
\author[{L.Aloui  and H.Azaza}]{{
\small  L.Aloui $^{1,2}$ and H.Azaza $^1$
\\
$^1$ LAMMDA-ESSTHS,  Universit\'e de Sousse, Tunisia. \\
$^2$ Universit\'e de Tunis El Manar, Tunisia. 
}}

\begin{document}
\begin{abstract}
In this paper we study the behavior of the total energy and the $L^2$-norm of solutions of  two coupled hyperbolic equations by velocities  in exterior domains. Only one of the two equations is directly damped by a localized damping term. We show that, when the damping set contains the coupling one and the coupling term is effective at infinity and on captive region, then the total energy decays uniformly and the $L^2$-norm of smooth solutions is bounded. In the case of two  Klein-Gordon equations with equal speeds we deduce an exponential decay of the energy.
	\end{abstract}			
	\thanks{Email : lassaad.aloui@fst.utm.tn.   \\ houdaazaza@gmail.com.}
\maketitle
	\date{}
\section{Introduction and statement of the results}
Let $\Omega$ be a domain of $\R^d$ ,$d \geqslant 2$. We denote by $\Delta$ the Laplace operator on $\Omega$ with Dirichlet boundary condition. We  consider the following hyperbolic equation with localized linear damping
\begin{equation} 
 \label{pbonde}
  \left\{
    \begin{aligned}
     &  \partial_t^2 u -\Delta u +mu+a(x)\partial_t u=0 & \text{ in } \R_+ \times \Omega,\\&
                 u=0 &  \text{ on }  \R_+ \times \Gamma, \\&
                 (u(0,.),\partial_t u(0,.))=(u_0,u_1)&  \ in \ \Omega, 
    \end{aligned}
  \right.
\end{equation}
where $ a\in L^\infty(\Omega)$ is a nonnegative smooth function    and $m \in \R_+ $. It is easy to verify that the energy given by \begin{align}
\label{eneg11}
E_u(t) = \frac{1}{2}\int_\Omega |\partial_t u(t,x)|^2 + |\nabla u(t,x)|^2 + m |u(t,x)|^2 \ dx, 
\end{align} 
is non-increasing and 
\begin{align*}
E_u(0)= \int_0^t\int_\Omega a(x)|\partial_t u(t,x)|^2\ dx dt+ E_u(t), \ t>0.
\end{align*}

When $m=0$, the stabilization problem for the linear damped wave equation has been studied by several authors. More precisely, when $\Omega$  is bounded,  the uniform  decay  of the total  energy is equivalent to the geometric control condition of Bardos et al \cite{e8}. On the other hand, if $\Omega$ is not bounded then, in general, the decay rate of the total energy  cannot be uniform. Indeed, in the  whole   space,i.e.  $ \Omega= \R^d$,  Matsumura \cite{e11} obtained a precise $L^p-L^q$ type decay estimate for solutions of $\eqref{pbonde}$, when $a(x)=1$,

\begin{equation}
E_u(t) \leqslant C(1+t)^{-1-d(\frac{1}{i}-\frac{1}{2})} I_i^2,
\end{equation}
\begin{equation}
\|u(t,.)\|_{L^2}^2 \leqslant C (1+t)^{d(\frac{1}{i}-\frac{1}{2})} I_i^2,
\end{equation}
where $C$ is a positive constant, $i \in [1,2 ]$ and $I_i^2 = \|u_0\|_{H^1}^2 + \|u_1\|_{L^2}^2 + \|u_0\|_{L^i}^2 + \|u_1\|_{L^i}^2$. The proof in \cite{e11} is based on a Fourier transform method. In the case  of  exterior  domains and when $a(x) \geqslant a^- >0$ on $\Omega$, it is easy to show that  the weak solution $u$  of  the system $\eqref{pbonde}$  satisfies
\begin{equation}
\label{0101}
E_u(t)\leqslant C(1+t)^{-1} I_2^2\ and \  \|u(t)\|_{L^2}^2 \leqslant C I_2^2, \  for  \  all \  t \geqslant 0.
\end{equation}
In \cite{e10}, Nakao obtained the estimate $\eqref{0101}$ for a damper which is positive near infinity and near a part of the boundary (Lions's condition).  Daoulatli  in \cite{e3} generalized this result by assuming  that each trapped ray meets the damping region which is also effective at infinity. Recently, Aloui et al \cite{e012} established the uniform stabilization of the total energy for the system $\eqref{pbonde}$  when the  initial data are compactly supported. They proved that the rate of decay turns out to be the same as those of the heat equation, which shows that  the effective damper at space infinity strengthens the parabolic structure in the equation. 
\\ 

In the case  $m>0$, the energy $\eqref{eneg11}$ contains the $L^2$ norm. Then, using the semi-group property, the type of decay $\eqref{0101}$ implies the expnential one
\begin{align}
E_u (t) \leqslant C e^{-\delta t} E_u(0), \ for \ all \ t \geqslant 0,
\end{align}
where $C , \delta$ positive constants. In \cite{e01} Zuazua considered the nonlinear Klein-gordon equations with dissipative term and he proved the exponential decay of energy through the weighted energy method. This result has been generalized by Aloui et al \cite{e03}  for  more general nonlinearities. We refer the reader to the works of Dehman et al \cite{e02} and Laurent et al \cite{e04} for related results. 
\\

In this paper  we will study  the stabilization problem for a system of two coupled hyperbolic equations on  exterior domain. More precisely, let $O$ be a compact domain of $\R^d$ with $\mathcal{C}^\infty$ boundary $\Gamma = \partial O$ and $\Omega =\R^d \backslash O$
 \begin{equation} 
 \label{pb311}
  \left\{
    \begin{aligned}
     &  \partial_t^2 u -\Delta u +m_1u+b(x)\partial_t v +a(x)\partial_t u=0 & \text{ in } \R_+ \times \Omega, \\ &
  \partial_t^2 v -\gamma^2\Delta v +m_2v-b(x) \partial_t u =0&\text{ in } \R_+ \times \Omega ,\\&
                 u=v=0 &  \text{ on }  \R_+ \times \Gamma, \\&
                 (u(0,.),\partial_t u(0,.))=(u_0,u_1)&\ in\ \Omega, \\&
                 (v(0,.),\partial_t v(0,.))=(v_0,v_1)& \ in \ \Omega,
    \end{aligned}
  \right.
\end{equation}
where $b\in L^\infty(\Omega) $ is  a smooth function, $m_1,m_2\in \R_+$  and  $\gamma$ is a positive constant.\\ 
 We associate to the system $\eqref{pb311}$ the energy functional given by

\begin{align*}
E_{u,v}(t)&= \frac{1}{2}\int_\Omega |\nabla u(t,x)|^2 +|\partial_t u(t,x)|^2  +m_1|u(t,x)|^2\ dx \\& + \frac{1}{2}\int_\Omega \gamma^2 |\nabla v(t,x)|^2 +|\partial_t v(t,x)|^2 + m_2|v(t,x)|^2 \ dx.
\end{align*}
 Let $\mathcal{H}= \Big(H_D^{1}(\Omega)\times L^2(\Omega)\Big)^2$ be the completion of $(C_0^\infty (\Omega))^4$ with respect to the norm 
\begin{align*}
\|(w_0,w_1,w_2,w_3)\|_{\mathcal{H}}= \Big( \int_\Omega |\nabla w_0|^2+ \gamma^2|\nabla w_2|^2+  m_1|w_0|^2 +m_2|w_2|^2   +|w_1|^2+|w_3|^2 \ dx \Big)^{\frac{1}{2}}.
\end{align*}
The linear evolution equation $\eqref{pb311}$ can be rewritten under the form
\begin{equation}
 \label{pb32}
  \left\{
    \begin{aligned}
     &  \mathcal{U}_t+ \mathcal{A}\mathcal{U}=0,\\&
     \mathcal{U}(0)= \mathcal{U}_0 \in \mathcal{H},
    \end{aligned}
  \right.
\end{equation}
where 
$$\mathcal{U}=\begin{pmatrix} u\\ \partial_t u \\ v \\ \partial_t v

\end{pmatrix}, \mathcal{U}_0= \begin{pmatrix} u_0\\ u_1 \\ v_0 \\ v_1

\end{pmatrix}$$
and the unbounded operator $ \mathcal{A}$ on $ \mathcal{H}$ with domain 
\begin{align*}
D(\mathcal{A})= \{\mathcal{U} \in \mathcal{H}, \mathcal{A}\mathcal{U} \in \mathcal{H}\}
\end{align*} is defined by
$$ \mathcal{A}= \begin{pmatrix}
0& -Id &0&0 \\ -\Delta +m_1Id& a & 0& b \\ 0& 0&0& -Id \\ 0& -b & -\gamma^2\Delta+m_2Id & 0
\end{pmatrix}.$$
From the linear semi-group theory,  we can infer that for $\mathcal{U}_0 \in \mathcal{H}$ the problem $\eqref{pb32}$ admits a unique solution $\mathcal{U}\in C^0([0,+\infty[, \mathcal{H})$. \\ In addition, if $\mathcal{U}_0 \in D(\mathcal{A}^n)$, for $n \in \N$, then the solution $\mathcal{U} \in  \displaystyle\bigcap_{i=0}^n C^{n-i} (\R_+, D(\mathcal{A}^i))$.
\\
It is easy to verify that
\begin{equation}
\label{decroissance}
\frac{d}{dt}E_{u,v}(t)= -\int_\Omega a(x)|\partial_t u(t,x)|^2\ dx.
\end{equation}
Thus $E_{u,v}(t)$ is decreasing with respect to time.
\\
In bounded domain and under some  geometric conditions, Kapitonov \cite {e91}  considered  the case  of equal speeds ($\gamma =1$) and proved  the uniform decay
\begin{equation}
\label{stability}
   E_ {u, v}(t) \leqslant Me ^ {- \beta t} E_ {u, v} (0), \ for \ all \ t \geqslant 0,
   \end{equation}
where $ M, \beta> 0 $. In \cite{e92}, Ammar et al  studied the indirect stability of system \eqref{pb311}  in the case of one-dimensional space and when $a$ and $b$ have disjoint supports. More precisely, they  established that the "classical" internal damping applied to only one of the equations never gives exponential stability if  $\gamma \neq 1$ and for the case  $\gamma =1$  they gave an explicit necessary and sufficient conditions for the stability to occur. In \cite{e93}, Toufayli generalized this result  for different speeds  and established, under some geometric conditions,  a polynomial stability.
\\

The problem of the indirect stabilization  has been also studied for coupled wave equations by displacements (weakly coupled). Indeed Alabau et al  \cite{e1} considered the following system
 \begin{equation}
 \label{pb3333}
  \left\{
    \begin{aligned}
     &  \partial_t^2 u(t,x) -\Delta u(t,x) +b(x) v(t,x) +a(x)\partial_t u(t,x)=0&\text{ in } \R_+ \times \Omega, \\ &
  \partial_t^2 v(t,x) -\Delta v(t,x) +b(x) u(t,x) =0&\text{ in } \R_+ \times \Omega ,\\&
                 u=v=0 &  \text{ on }  \R_+ \times \Gamma, \\&
                 (u(0,.),\partial_t u(0,.))=(u_0,u_1)& \ in \ \Omega, \\&
                 (v(0,.),\partial_t v(0,.))=(v_0,v_1)& \ in\ \Omega,
    \end{aligned}
  \right.
\end{equation}
where $\Omega$ is a bounded domain. They proved that the system $\eqref{pb3333}$ can not be exponentially stable and when the coupling term is  constant they established a  polynomial decay. In \cite{e2} Alabau et al improved this result  by assuming that the regions $\{a > 0\}$ and $\{b > 0\}$ both verify GCC and the coupling term satisfies a smallness assumption. This result has been generalized by Aloui et al \cite{e13}, for more natural smallness condition on the infinity norm of the coupling term. 
 Recently, Daoulatli \cite{e15} showed  that the rate of energy  decay  for solutions to the system  on a compact manifold with a boundary is determined from a first order differential equation when  the   coupling zone  and the damping zone verify the GCC.  
\\

In the sequel, we fix a constant $R_0 > 0$ such that $$O \subset B_0 = \{x \in \R^d, |x| < R_0\}.$$ Suppose that there exist two positive constants $a^-$ and $b^-$ such that the damping set $\omega_a :=\{a(x) >a^->0\}$  and the coupling set $\omega_b:=\{b(x)> b^->0\}$  are non-empty open subsets of $\Omega$. As usual for damped wave (resp. Klein-Gordon) equations, we have
to make some geometric assumptions on the sets $\omega_a$ and $\omega_b$ so that the energy of a single wave decays sufficiently rapidly at infinity. Here, we shall use the Geometric control condition.
\begin{definition}
\label{def}
(see \cite{e8,e9}) We say that a set  $\omega$ of $\Omega$ satisfies the geometric control condition \textbf{GCC} if there exists $T >0$ such that from every point in $\Omega$ the generalized geodesic meets the set $\omega$ in a time $t <T$.
\end{definition}
If $\omega$ satisfies \textbf{GCC}, we set
 $$ T_\omega= \inf\{T>0, (\omega,T) \text{ satisfies \textbf{GCC}}\}.$$
We need also the following assumptions
\begin{itemize}
\item[$(\mathcal{A}_1)$] $ supp(b) \ \subset  \ supp(a)$.


\item[$(\mathcal{A}_2)$] There exists $ R_1> R_0 $ such that

\begin{itemize}
\item[$\bullet$] $B_{R_1}^c \subset \omega_a\cap  \omega_b$,   if $\ (m_1,m_2)\in \R_+\times \R_+^*$, 
\item[$\bullet$]  $B_{R_1}^c\subset \omega_b$ and $a(x)=\beta b(x)$,  $\ |x| \geqslant R_1$, for some $\beta>0$,  if $\  m_1=m_2=0$.

\end{itemize}
\end{itemize}

\

 For $\gamma \in \R_+^*$, we set $$I_{\gamma}^2 =  E_{u,v}(0)  + (1-\frac{1}{\gamma^2})^2 E_{\partial_t u ,\partial_t v}(0) +\|(u,v)(0)\|_{L^2(\Omega)}^2 $$   and $$
 \mathcal{H}_\gamma=\begin{cases} \mathcal{H}\cap (L^2(\Omega))^4,  \ if  \ \gamma =1,  \\   D(\mathcal{A})\cap ( L^2(\Omega))^4 ,  \ if \  \gamma \neq 1. \end{cases}$$
\\ With this notation, we can state the stability result for  the system $\eqref{pb311}$. 
\begin{theorem}
\label{theoreme1}
Let $\gamma \in  \R_+^*$ and  $(m_1,m_2) \in  \lbrace (0,0)\rbrace  \cup \R_+\times \R_+^* $. We assume that $\omega_b$ satisfies the \textbf{GCC} and that the assumptions $(\mathcal{A}_1)$ and $(\mathcal{A}_2)$ hold. Then for any solution  $(u,v)$ of the system $\eqref{pb311}$ with initial data $(u_0,u_1, v_0,v_1) \in \mathcal{H}_\gamma$, we have 
\begin{equation}
\label{uv}
E_{u,v}(t) \leqslant C(1+t)^{-1}I_{\gamma}^2 \text{ and}\quad \|(u,v)(t)\|_{L^2}^2 \leqslant CI_{\gamma}^2,    \text{ for all }  t \geqslant 0,
\end{equation}
where $C$ is positive constant. In addition for $(u_0,u_1, v_0,v_1) \in \mathcal{H} $, $E_{u,v}(t)$ converges to zero as $t$ goes to infinity.
\end{theorem}

In the case of  Klein-Gordon-type systems  we obtain the following uniform decay.
\begin{corollary}
\label{corol}
 Let $m_1, m_2 , \gamma \in \R_+^*$.  Assume that $\omega_b$ satisfies the \textbf{GCC} and the assumptions $(\mathcal{A}_1)$  and $(\mathcal{A}_2)$ hold.
\begin{itemize}
\item[$\triangleright$] If $\gamma =1$,  then there exist positive constants $C$ and $\alpha$ such that 
 \begin{equation}
\label{theorem 2}
E_{u,v}(t) \leqslant Ce^{-\alpha t} E_{u,v}(0) ,    \text{ for all }  t \geqslant 0, 
\end{equation}
for all solution  $(u,v)$ of the  system $\eqref{pb311}$ with initial data $(u_0,u_1, v_0,v_1) \in \mathcal{H}_1$. 
\item[$\triangleright$] If $\gamma \neq 1$, then there exists a positive constant $C$  such that 
 \begin{equation}
\label{theorem 22}
E_{u,v}(t) \leqslant \frac{C}{t^n} \sum_{k=0}^n  E_{\partial_t^k u,\partial_t^k v}(0) ,    \text{ for all }  t \geqslant 0, 
\end{equation}
for all solution  $(u,v)$ of the  system $\eqref{pb311}$ with initial data $(u_0,u_1, v_0,v_1) \in D(\mathcal{A}^n)$.

\end{itemize} 
  
\end{corollary}

\begin{remark}
\begin{itemize}
\item[$\bullet$] To our best knowledge, our result is new for the indirect stabilization problem in exterior domains.  
\item[$\bullet$]  Remark that, when $\gamma =1$, the energy of the system $\eqref{pb311}$ decays as fast as that of the corresponding scalar damped equation.  So the coupling through velocities, in this case,  allows a full transmission of the damping effects, quite different from the coupling through the displacements.
\item[$\bullet$]  To prove our main result we study the energy first at infinity ( Section 2) and then in bounded regions (Section 3). Keeping, only the second step, we can obtain the expnential energy decay for the system $\eqref{pb311}$ in bounded domains with Dirichly boundary condition.
\item[$\bullet$]  Due to technical  difficulties  we did not cover the  Klein-Gordon-Wave case ($m_1>0$, $m_2=0$); we will be interested in the forthcoming work.
\end{itemize} 
\end{remark} 
We conclude this introduction with an outline of the rest of this paper. In Section 2 we estimate the total energy  at infinity by multiplier  arguments.  Section 3 is devoted to the study of the energy in bounded domain. The proof of this result is based on observability estimate for scalar wave equation. In order to control the compact terms, we prove in section 4 a weak observability estimate that is based on a unique continuation result. Finally, in Section 5 we combine the results of the previous sections to established our main results.
 \\

We denote by $ \Omega_R:= \Omega \cap B_R$ , $C_{R,R'}= \Omega \cap (B_{R'}\backslash B_R),$ when $0< R < R'$ ,
\begin{align*}
E^R(u,v,t) &=\frac{1}{2}\int_{|x|> R} |\partial_t u(t,x)|^2 +|\nabla u(t,x)|^2 +m_1 |u(t,x)|^2\ dx \\&+\frac{1}{2}\int_{|x|> R} |\partial_t v(t,x)|^2 +\gamma^2|\nabla v(t,x)|^2  +m_2 |v(t,x)|\ dx, 
\end{align*}
\begin{align*}
E_R(u,v,t) &=\frac{1}{2}\int_{\Omega_R} |\partial_t u(t,x)|^2 +|\nabla u(t,x)|^2 +m_1 |u(t,x)|^2\ dx \\&
+\frac{1}{2}\int_{\Omega_R} |\partial_t v(t,x)|^2 +\gamma^2|\nabla v(t,x)|^2  +m_2 |v(t,x)|\ dx, 
\end{align*}
and  $ A \lesssim B$ means $A \leqslant CB$ for some positive constante $C$.
\section{Estimate of energy near infinity}
The main result of this section is as follows. 
\begin{proposition}
\label{prop1}
Let $\gamma \in  \R_+^*$ and  $(m_1,m_2) \in  \lbrace (0,0)\rbrace  \cup \R_+\times \R_+^* $. Let $R_1>0$ be  such that $(\mathcal{A}_2)$ is satisfied and $R_2>R_1$. Then for every $\varepsilon >0$, there exists $C_\varepsilon >0$ such that for all  solution $(u,v)$  of $\eqref{pb311}$ with initial data $(u_0,u_1,v_0,v_1)\in \mathcal{H}_\gamma$, we have 
 \begin{align*}
 \|(u,v)(t)\|_{L^2(|x|>R_2)}^2 + \int_0^{t}E^{R_2}(u,v,s) ds  \lesssim  C_\varepsilon( E_{u,v}(0)+(1-\frac{1}{\gamma^2})^2 E_{\partial_t u, \partial_t v}(0))
 \end{align*}
 \begin{equation}
 \label{enex1g}
 +\varepsilon  \int_0^{t}  E_{u,v}(s) \ ds +  C_\varepsilon\int_0^{t}\int_{\Omega_{R_2}} |u|^2+|v|^2  \ dx ds + \|(u,v)(0)\|_{L^2(\Omega)}^2,
\end{equation}
for all $ t > 0$. 
\end{proposition}
 
Let $  \varphi  \in C^\infty (\R^d)$ be a function  satisfying $ 0  \leqslant  \varphi \leqslant 1 $ and 
$$ \varphi(x) =\begin{cases} 1 & \text{ for  } |x| \geqslant R_2  \\ 0 & \text{ for }  |x| \leqslant R_1. \end{cases} \ $$

To prove Proposition $\ref{prop1}$, we need the following Lemma.
\begin{lemma}
\label{lemme1}
 We assume the hypothesis of Proposition \ref{prop1} and  we consider $\varphi$ as above. Then for every $\varepsilon >0$, there exist $C_\varepsilon >0$ such that for all  solution $(u,v)$  of $\eqref{pb311}$ with initial data $(u_0,u_1,v_0,v_1)\in \mathcal{H}_\gamma$, we have 
\begin{align*}
\int_0^{t}\int_\Omega b(x)&\varphi|\partial_t v|^2 \ dx ds  \lesssim C_\varepsilon( E_{u,v}(0)+(1-\frac{1}{\gamma^2})^2 E_{\partial_t u, \partial_t v}(0))
\end{align*}
\begin{equation}
\label{Lemg}
+C_\varepsilon \int_0^{t}\int_{\Omega_{R_2}}|v|^2 \ dxds + \varepsilon \int_0^{t}E_{u,v}(s)\  ds   ,
\end{equation}
for all $t>0$. 
\end{lemma}

\begin{proof}[Proof of Lemma $\ref{lemme1}$]
Multiplying the first and the second equation of $\eqref{pb311}$ respectively by $ \varphi \partial_t v$ and $ \frac{1}{\gamma ^2}\varphi \partial_t u$ and integrating the sum of these results on $[0,t]\times \Omega$, we obtain
\begin{align*}
\Big[\int_\Omega & \frac{1}{\gamma ^2}\varphi \partial _tu  \partial_t v  +m_1 \varphi u v \ dx \Big]_0^{t} + \int_0^{t}\int_\Omega b(x) \varphi |\partial_t v|^2\ dx ds \\&= \int_0^{t}\int_\Omega \frac{1}{\gamma ^2} a(x) \varphi |\partial_t u|^2  - \varphi  \partial_t u \partial_t v  + \varphi \Delta u \partial_t v \\&  +(m_1-\frac{m_2}{\gamma ^2}) \varphi v\partial_t u + \varphi \Delta v \partial_t u    - (1-\frac{1}{\gamma ^2}) \varphi\partial_t v \partial_t^2 u\ dx ds.   
\end{align*}
Note that
\begin{align*}
\int_0^{t}\int_\Omega \varphi \Delta u \partial_t v& \ dx ds = \Big[\int_\Omega  \varphi  \Delta u v \ dx \Big]_0^{t} -\int_0^{t}\int_\Omega \varphi \Delta \partial_t u v \ dx ds \\& = -\Big[\int_\Omega    \nabla u(\nabla \varphi v + \varphi \nabla v)\ dx \Big]_0^{t} -\int_0^{t}\int_\Omega \Delta (\varphi v) \partial_t u  \ dx ds  \\&=  -\int_0^{t}\int_\Omega( \Delta \varphi v + \Delta v \varphi + 2 \nabla v \nabla \varphi ) \partial_t u  \ dx ds 
\end{align*}
\begin{equation}
\label{cro}
-\Big[\int_\Omega    \nabla u(\nabla \varphi v + \varphi \nabla v)\ dx \Big]_0^{t}.
\end{equation}
Then using  Young's inequality, we get
\begin{align*}
\Big[F_\gamma \Big]_0^{t} + \int_0^{t}\int_\Omega b(x)\varphi|\partial_t v|^2\ dx ds & \lesssim \int_0^{t}\int_\Omega   ((\frac{1}{\gamma^2} a(x)+2)\varphi + C_\varepsilon|\nabla \varphi|^2) |\partial_t u|^2 \\&+ C_\varepsilon \varphi (1-\frac{1}{\gamma^2})^2 |\partial_t^2 u|^2 + |\Delta \varphi|^2|v|^2 \ dx ds \\& + \varepsilon \int_0^{t}\int_\Omega |\nabla v|^2 +(m_1-\frac{m_2}{\gamma ^2})^2\|\varphi\|_\infty |v|^2\\& + |\partial_t u|^2+ \|\varphi\|_\infty|\partial_t v |^2 \ dx ds,
\end{align*}
where
\begin{align*}
F_\gamma=\int_\Omega \varphi(\frac{1}{\gamma ^2} \partial _t u  \partial_t v  +m_1 uv) +  \nabla u(\nabla \varphi v + \varphi \nabla v)\ dx.
\end{align*}
By hypothesis
\begin{align}
\label{2.5}
 supp(\varphi) \subset  \{x\in \Omega, a(x) >a^-\},
\end{align}
so, we deduce that
 \begin{align*}
  \Big[F_\gamma \Big]_0^{t} + \int_0^{t}\int_\Omega b(x)\varphi|\partial_t v|^2\ dx ds  \lesssim C_\varepsilon\int_0^{t}\int_\Omega  a(x) ( |\partial_t u|^2 
 \end{align*}
 \begin{equation}
 \label{*0}
+  (1-\frac{1}{\gamma^2})^2 |\partial_t^2 u|^2) \ dx ds + \int_0^t\int_{\Omega_{R_2}} |v|^2 \ dx ds + \varepsilon \int_0^t E_{u,v}(s) \ ds.
 \end{equation}
 Using  the energy decay  $\eqref{decroissance}$ and the fact that $(m_1,m_2) \in\{(0,0)\}\cup \R_+\times \R_+^*$, we can see that
 \begin{align}
&\label{1} \Big|F_\gamma (s)\Big| \lesssim E_{u,v}(s) \lesssim E_{u,v}(0),\quad \forall \ s\geqslant 0.
 \end{align}
Combining $\eqref{decroissance}$, $\eqref{*0}$ and $\eqref{1}$, we obtain $\eqref{Lemg}$. 
 
\end{proof}
\begin{lemma}
\label{lemme}
Let $\gamma \in  \R_+^*$ and $(m_1,m_2)=(0,0)$.  Let $R_1>0$ be  such that $(\mathcal{A}_2)$ is satisfied and $R_2>R_1$. Then for every $\varepsilon >0$, there exists $C_\varepsilon >0$ such that for all  solution $(u,v)$  of $\eqref{pb311}$ with initial data $(u_0,u_1,v_0,v_1)\in \mathcal{H}_\gamma$, we have 
\begin{align*}
  \|(u,v)(t)\|_{L^2(|x|>R_2)}^2 + \int_0^{t}E^{R_2}(v,s)   ds  \lesssim C_\varepsilon(  E_{u,v}(0)+(1-\frac{1}{\gamma^2})^2 E_{\partial_t u, \partial_t v}(0))
 \end{align*}
 \begin{equation}
 \label{envv}
 +\varepsilon  \int_0^{t}  E_{u,v}(s) \ ds + C_\varepsilon (\int_0^{t}\int_{\Omega_{R_2}} |u|^2+|v|^2  \ dx ds + \|(u,v)(0)\|_{L^2(\Omega)}^2),
\end{equation}
for all $t >0$.
\
 Where $ E^{R_2}(v,t)  =\frac{1}{2}\int_{|x|>R_2} |\partial_t v(t,x)|^2+ |\nabla v(t,x)|^2 \ dx.$
\end{lemma}
\begin{proof}[Proof of Lemma $\ref{lemme}$]
 We write the system $\eqref{pb311}$ in the form 

\begin{equation}
\label{sysex11}
  \left\{
    \begin{aligned}
     &  \partial_t^2 u -  \Delta u+ \frac{a(x)}{b(x)} \partial_t^2 v -  \frac{a(x)}{b(x)} \gamma^2\Delta v +  b(x) \partial_t v =0   & \text{ in } \R_+ \times \Omega_{R_1^c},\\ &
  -\partial_t^2 v  +\gamma^2\Delta v +b(x) \partial_t u  = 0  & \text{ in } \R_+ \times \Omega_{R_1^c}.
      \end{aligned}
  \right.
\end{equation}
Multiplying the first equation of $\eqref{sysex11}$ by $  \varphi v $ and the second  one by $\frac{1}{\gamma ^2}\varphi u $ and integrating the sum of these results on $[0,t]\times \Omega$, we obtain
\begin{align*}
  \int_\Omega&\frac{\varphi b(x)}{2}(\frac{1}{\gamma^2}|u(t)|^2 + |v(t)|^2) \ dx +  \beta\int_0^{t} \int_\Omega \varphi(|\partial_t v|^2 + |\nabla v|^2) \ dx  ds  \\&  =  \int_0^{t}\int_\Omega 2\varphi \beta|\partial_t v|^2  + \frac{ \gamma ^2\beta\Delta\varphi }{2}|v|^2- \nabla u( \nabla \varphi v +  \varphi\nabla v) \\&+ \nabla v(\nabla \varphi u + \varphi \nabla u)   +(1-\frac{1}{\gamma ^2})\varphi \partial_t u \partial_t v \ dx ds \\&+ \int_\Omega \frac{\varphi b(x)}{2}(\frac{1}{\gamma^2}|u(0)|^2 + |v(0)|^2) \ dx - \Big[G_\gamma \Big]_0^{t}  \ dx,
\end{align*}
where 
\begin{align*}
G_\gamma= \int_\Omega \varphi(\partial_t u v + \partial_t v v -\frac{1}{\gamma^2}\partial_t v u )\ dx .
\end{align*}
According to Lemma $\ref{lemme1}$, hypothesis $(\mathcal{A}_2)$ and using Young's inequality, we deduce that
\begin{align*}
\int_\Omega& \varphi (|u(t)|^2 + |v(t)|^2) \ dx + \int_0^{t}  \int_\Omega \varphi(|\partial_t v|^2 + |\nabla v|^2) \ dx   ds \\&\lesssim E_{u,v}(0) + (1-\frac{1}{\gamma^2})^2 E_{\partial_t u,\partial_t v}(0) +\|(u,v)(0)\|_{L^2}^2
\end{align*}
\begin{equation}
+ \int_0^{t}\int_{\Omega_{R_2}}|v|^2+ |u|^2\ dx ds  + \varepsilon  \int_0^{t} E_{u,v}(s) \ ds -\Big[G_\gamma\Big]_0^t.
\end{equation}
But we have
\begin{align*}
\Big|G_\gamma(t)\Big| &\lesssim E_{u,v}(t) + \varepsilon_1 \int_\Omega \varphi (|u(t)|^2 + |v(t)|^2) \ dx \\& \lesssim  E_{u,v}(0) + \varepsilon_1 \int_\Omega \varphi( |u(t)|^2 + |v(t)|^2) \ dx.
\end{align*}
So, for $\varepsilon_1$ small enough we get 
\begin{align*}
\int_\Omega& \varphi (|u(t)|^2 + |v(t)|^2) \ dx + \int_0^{t}  \int_\Omega \varphi(|\partial_t v|^2 + |\nabla v|^2) \ dx  \  ds \\&\lesssim E_{u,v}(0) + (1-\frac{1}{\gamma^2})^2 E_{\partial_t u,\partial_t v}(0) +\|(u,v)(0)\|_{L^2}^2
\end{align*}
\begin{equation}
\label{vgamma}
+ \int_0^{t}\int_{\Omega_{R_2}}|v|^2+ |u|^2\ dx ds  + \varepsilon  \int_0^{t} E_{u,v}(s) \ ds.
\end{equation}
Since 
\begin{align}
\label{varphi}
\varphi \equiv 1 \ for  \ |x| \geqslant R_2
\end{align}
we deduce that
\begin{align*}
&\int_{|x|>R_2} |u(t)|^2 +|v(t)|^2 \ dx + \int_0^t E^{R_2}(v,s) \ ds \\&\leqslant  \int_\Omega \varphi (|u(t)|^2 + |v(t)|^2) \ dx + \int_0^{t}  \int_\Omega \varphi(|\partial_t v|^2 + |\nabla v|^2) \ dx  ds.
\end{align*}
Combining this estimate with $\eqref{vgamma}$, we conclude $\eqref{envv}$. This finishes the proof of Lemma $\ref{lemme}$. 
\end{proof}
Now we give the proof of Proposition $\ref{prop1}$.
\begin{proof}[Proof of Proposition $\ref{prop1}$]
We distinguish the case $m_1=m_2=0 $ and  the case where $m_1 \in \R_+$ and $m_2\in \R_+^*$.
\\\textbf{First case $m_1=m_2=0$}. Multiplying the first equation of $\eqref{pb311}$ by $ \varphi u $ and integrating on $[0,t]\times \Omega$, we obtain
\begin{align*}
\Big[\int_\Omega&  \varphi (\partial_t u u + \frac{a(x)|u|^2}{2}  + b(x) uv) \ dx \Big]_0^{t} + \int_0^{t} \int_\Omega \varphi ( |\nabla u|^2 + |\partial_t u|^2 )\  dxds  
\end{align*}
\begin{equation}
\label{ugamma}
= \int_0^{t}\int_\Omega  2\varphi |\partial_t u|^2  + \frac{\Delta \varphi}{2} |u|^2  + \varphi b(x) v\partial_t u \ dx ds.
\end{equation}
Note that we have
\begin{align*}
\int_0^{t}\int_\Omega &\varphi  b(x) v \partial_t u \ dx ds = \int_0^{t}\int_\Omega \varphi v(\partial_t^2 v - \gamma^2\Delta v )  \ dx ds 
\end{align*}
\begin{equation}
\label{eq4}
= \Big[\int_\Omega \varphi \partial_t v v \ dx \Big]_0^{t} + \int_0^{t}\int_\Omega  \varphi (\gamma ^2|\nabla v|^2-  |\partial_t v|^2)-\gamma^2\frac{\Delta \varphi}{2}|v|^2\ dx ds. 
\end{equation}
So, combining this identity  with $\eqref{ugamma}$  and using $\eqref{2.5}$, we get
\begin{align*}
 \int_0^{t} \int_\Omega& \varphi(|\partial_t u|^2 + |\nabla u|^2) \ dx  ds \lesssim \int_0^{t}\int_\Omega  a(x) |\partial_t u|^2   + \int_0^{t} \int_\Omega \varphi(|\partial_t v|^2\\& + |\nabla v|^2) \ dxds +  \int_0^{t}\int_{\Omega_{R_2}} |u|^2 +|v|^2 \ dx ds 
\end{align*}
\begin{equation}
\label{*2**}
-\Big[\int_\Omega \varphi( \partial_t u u + b(x) uv +\frac{a(x)|u|^2}{2} -\partial_t v v ) \  dx\Big]_0^{t}. 
\end{equation}
Using that,
\begin{align*}
&\Big|\int_\Omega \varphi( \partial_t u u + b(x) uv +\frac{ a(x)|u|^2}{2} -\partial_t v v )(t) \  dx\Big| \\&\lesssim  C_\varepsilon E_{u,v}(0) + \int_\Omega \varphi (|u(t)|^2 + |v(t)|^2 ) \ dx \\&  \Big|\int_\Omega \varphi( \partial_t u u +b(x) uv +\frac{a(x)|u|^2}{2} -\partial_t v v )(0) \  dx\Big| \\&\lesssim E_{u,v}(0)+\|(u,v)(0)\|_{L^2}^2, 
\end{align*}
we obtain
\begin{align*}
 \int_0^{t} \int_\Omega\varphi(|\partial_t u|^2 + |\nabla u|^2) \ dx  ds \lesssim   C_\varepsilon E_{u,v}(0) + \int_\Omega \varphi (|u(t)|^2 + |v(t)|^2 ) \ dx 
\end{align*}
\begin{equation}
  + \int_0^{t} \int_\Omega\varphi(|\partial_t v|^2 + |\nabla v|^2) \ dx \ ds +  \int_0^{t}\int_{\Omega_{R_2}} |u|^2 +|v|^2 \ dx ds +\|(u,v)(0)\|_{L^2}^2.
\end{equation}
According to $\eqref{vgamma}$ and using $\eqref{varphi}$, we get
\begin{align*}
 \int_0^{t} E^{R_2}(u,s) ds \lesssim   C_\varepsilon E_{u,v}(0)   + (1-\frac{1}{\gamma^2})^2 E_{\partial_t u,\partial_t v}(0)  + \varepsilon\int_0^{t} E_{u,v}(s)\ ds 
\end{align*}
\begin{equation}
\label{*22**}
+  \int_0^{t}\int_{\Omega_{R_2}} |u|^2 +|v|^2 \ dx ds +\|(u,v)(0)\|_{L^2}^2,
\end{equation}
where $E^{R_2}(u,t) = \frac{1}{2}\int_{|x|>R_2} |\partial_t u(t,x)|^2 + |\nabla u(t,x)|^2\ dx$.
\\ Combining $\eqref{envv}$ and $\eqref{*22**}$, we conclude $\eqref{enex1g}$.
\\\textbf{Second case $m_1\in \R_+$ and $m_2\in \R_+^*$}. Multiplying the first and the second equation of $\eqref{pb311}$ respectively by $ \varphi u$ and $ \varphi v$ and integrating the sum of these results on $[0,t]\times \Omega$, we obtain
\begin{align*}
 \int_\Omega \varphi& \frac{ a(x)|u(t)|^2}{2} \ dx  + \int_0^{t} \int_\Omega\varphi(|\partial_t u|^2 + |\nabla u|^2 +m_1 |u|^2+|\partial_t v|^2 \\&+ |\nabla v|^2 +m_2 |v|^2 ) \ dx  ds  = \int_0^{t}\int_\Omega  2\varphi ( |\partial_t u|^2+ |\partial_t v|^2 )\ dx ds \\&+ \int_0^{t} \int_\Omega \frac{\Delta \varphi}{2}(|u|^2+ \gamma^2|v|^2) +2\varphi b(x)  v \partial_t u  \ dx ds \\&-\Big[\int_\Omega \varphi(\partial_t  u u + \partial_t v v+ b(x) uv)\ dx \Big]_0^{t} +  \int_\Omega \varphi \frac{ a(x)|u(0)|^2}{2}\ dx\\&\lesssim \int_0^{t}\int_\Omega  a(x)|\partial_t u|^2 + \varphi |\partial_t v|^2 + \varepsilon \|\varphi\|_\infty  |v|^2 \ dx ds \\&-\Big[\int_\Omega \varphi(\partial_t  u u + \partial_t v v+ b(x) uv)\ dx \Big]_0^{t} +  \int_\Omega \varphi \frac{a(x)|u(0)|^2}{2}\ dx
\end{align*}
\begin{equation}
\label{eq31}
  + \int_0^{t}\int_{\Omega_{R_2}} |u|^2+|v|^2 \ dx ds.
\end{equation}
Using the following estimates  for $\varepsilon_2$ small enough
\begin{align*}
&\Big|\int_\Omega \varphi((\partial_t  u u + \partial_t v v+ b(x)uv)(t))\ dx \Big| \lesssim E_{u,v}(0) + \varepsilon_2 \int_\Omega \varphi |u(t)|^2 \ dx,\\& \Big|\int_\Omega \varphi((\partial_t  u u + \partial_t v v+b(x) uv)(0))\ dx \Big| \lesssim E_{u,v}(0) + \|u(0)\|_{L^2}^2,
\end{align*}
and according to Lemma $\ref{lemme1}$, we infer $\eqref{enex1g}$.
The proof of proposition $\ref{prop1}$ is now completed.
\end{proof}
  \section{Estimate of energy in bounded region}
In this section, we will study the energy  in bounded domain. For this aim, we consider a function $  \psi  \in C_0^\infty (\R^d)$ such that $ 0  \leqslant  \psi \leqslant 1$ and 
$$ \psi(x) =\begin{cases} 1 & \text{ for  } |x| \leqslant R_3  \\ 0 & \text{ for }  |x| \geqslant R_4. \end{cases} $$
where $R_4> R_3 > R_1$ and $R_1>0$ be  such that $(\mathcal{A}_2)$ is satisfied.

It is easy to verify that $(u^i,v^i)=(\psi u, \psi v)$  satisfies the following system 
\begin{equation}
\label{sysin1}
  \left\{
    \begin{aligned}
     &  \partial_t^2 u^i -\Delta u^i  +m_1u^i+b(x)\partial_t v^i +a(x)\partial_t u^i=- 2\nabla \psi \nabla u - u \Delta \psi\  \text{in } \R_+\times \Omega_{R_4}  \\&
  \partial_t^2 v^i - \gamma^2\Delta v^i +m_2v^i-b(x)\partial_t u^i = -2\gamma^2\nabla \psi \nabla v- \gamma^2 v \Delta \psi\qquad  \text{in } \R_+\times \Omega_{R_4}  \\&
                 u^i=v^i=0 \qquad \qquad \qquad  \qquad \qquad  \qquad\qquad  \qquad\qquad  \qquad\quad \text{ on }  \R_+ \times \partial\Omega_{R_4} \\&
                 (u_0^i,u_1^i, v_0^i, v_1^ i)=  (\psi u_0,\psi u_1, \psi v_0, \psi v_1).                 
    \end{aligned}
  \right.
\end{equation}
\begin{proposition}
\label{prop2}
Let  $\gamma \in  \R_+^*$,  $(m_1,m_2) \in  \lbrace (0,0)\rbrace  \cup \R_+\times \R_+^* $ and $\psi$ be as above. Assume  that the assumption $(\mathcal{A}_1)$  holds and that $(\omega_b,T)$ geometrically controls $ \Omega$ for some $T>0$. 
Then  for every $\varepsilon >0$, there exist $C_\varepsilon >0$ such that for all  solution $(u,v)$  of $\eqref{pb311}$ with initial data $(u_0,u_1, v_0, v_1)\in \mathcal{H}_\gamma$, we have 

 \begin{align*}
 \int_t^{t+T}E_{R_3} (u,v,s) ds  \lesssim   C_\varepsilon \int_t^{t+T} \int_\Omega a(x) (|\partial_t u|^2+ (1-\frac{1}{\gamma^2})^2 |\partial_t^2 u|^2) \ dx ds 
 \end{align*}
 \begin{equation}
 \label{eningg}
+\varepsilon \int_t^{t+T} E_{u,v}(s) \ ds  + C_\varepsilon \int_t^{t+T}\int_{\Omega_{R_4}}| u|^2+| v|^2 \ dx ds +C_\varepsilon \int_t^{t+T} E^{R_3}(u,v,s)ds  + \Big[\mathcal{K}_\gamma\Big]_t^{t+T}
\end{equation}
for all $ t> 0$. Where
\begin{align*}
\mathcal{K}_\gamma= - \int_\Omega \frac{b(x)}{\gamma^2} \partial_t u^i \partial_tv^i + \nabla u^i \nabla((b(x) v^i)  +m_1ab(x) u^i v^i \ dx.
\end{align*} 
\end{proposition}
In order to prove proposition $\ref{prop2}$ we need the following result. 
\begin{lemma}
\label{lemme2}
Assume that the hypothesis of Proposition \ref{prop2} hold. Then for every $\varepsilon >0$, there exists $C_\varepsilon >0$ such that for all  solution $(u,v)$  of $\eqref{pb311}$ with initial data $(u_0,u_1, v_0, v_1)\in \mathcal{H}_\gamma$, we have 
\begin{align*}
\int_t^{t+T} \int_{\Omega}b(x)^2|\partial_t v^i|^2& \ dx ds \lesssim C_\varepsilon \int_t^{t+T}\int_\Omega a(x)(|\partial_t u|^2+ (1-\frac{1}{\gamma^2})^2|\partial_t^2 u|^2)  \ dxds \\&+ \varepsilon \int_t^{t+T} E_{u,v}(s) \ ds+ C_\varepsilon \int_t^{t+T}\int_{\Omega_{R_4}}| v|^2+ |  u|^2\ dx ds 
\end{align*}
\begin{equation}
\label{Lem6}
+ C_\varepsilon \int_t^{t+T}\int_{C_{R_3,R_4}}| \nabla u|^2+ |  \nabla v|^2\ dx ds+ \Big[\mathcal{K}_\gamma\Big]_t^{t+T} ,
\end{equation}
for all $t >0$. 
\end{lemma}
\begin{proof}[proof of Lemma $\ref{lemme2}$ ]
We multiply the first and the second equation of $\eqref{sysin1}$ respectively by $ b(x)\partial_t v^i$ and $\frac{b(x)}{\gamma^2}\partial_t u^i$ and  we integrate the sum of these results on $[t,t+T]\times \Omega$, we get
 
\begin{align*}
\Big[\mathcal{K}_\gamma \Big]_t^{t+T} +& \int_t^{t+T}\int_\Omega  b^2(x)|\partial_t v^i|^2 \ dx ds = \int_t^{t+T}\int_\Omega  \frac{b^2(x)}{\gamma^2} |\partial_t u^i|^2 - ab(x) \partial_tu^i \partial_t v^i  
\\&+(m_1-\frac{m_2}{\gamma^2})b(x) v^i \partial_t u^i ) \ dx ds -\int_t^{t+T}\int_\Omega b(x) (2\nabla u \nabla \psi + \Delta \psi u) \partial_t v^i\\& +\frac{b(x)}{\gamma^2}(2\nabla v \nabla \psi + \Delta \psi v) \partial_t u^i \ dx ds  +\int_t^{t+T}\int_\Omega  (\frac{1}{\gamma^2}-1)b(x) \partial_t^2 u^i\partial_tv^i \ dxds\\&  -\int_t^{t+T}\int_\Omega \partial_t u^i (\Delta b(x) v^i + 2\nabla b(x)\nabla v^i)\ dx ds.
\end{align*}
From Young's inequality and using hypothesis ($\mathcal{A}_1$), we infer that
\begin{align*}
&\Big[\mathcal{K}_\gamma \Big]_t^{t+T} + \int_t^{t+T}\int_\Omega  b^2(x)|\partial_t v^i|^2 \ dx ds \\&\lesssim C_\varepsilon\int_t^{t+T}\int_\Omega  a(x) (|\partial_t u|^2 +(1-\frac{1}{\gamma^2})^2 |\partial_t^2 u|^2) \ dx ds  
\\&+ \varepsilon \int_t^{t+T}\int_\Omega (m_1-\frac{m_2}{\gamma^2})^2 |v|^2   + |\partial_t u|^2 + |\partial_t v|^2 +|\nabla v|^2 \ dx ds  
\end{align*}
\begin{equation}
\label{GC}
+ C_\varepsilon\int_t^{t+T}\int_{\Omega_{R_4}} |u|^2+|v|^2 \ dx ds  +C_\varepsilon\int_t^{t+T}\int_{C_{R_3,R_4}} |\nabla u|^2+| \nabla v|^2 \ dx ds.
\end{equation}

This implies $\eqref{Lem6}$.

\end{proof}
\begin{proof}[Proof of proposition $\ref{prop2}$]
First, we recall  the following observability estimate for the  wave equation  ( see  proposition $3$, \cite{e3}).
\begin{lemma}
\label{obs}
 Let $\gamma , T > 0$ and $\mathcal{O} $ a bounded domain.  Let $ \phi$ be a nonnegative function on $\mathcal{O}$ and setting
 $$ \mathcal{V}=\{\phi(x) >0\}.$$
 We assume that $(\mathcal{V},T)$ satisifies the \textbf{GCC}. There exists $C_T>0$, such that for all $(u_0,u_1)\in H_0^1 (\mathcal{O})\times L^2(\mathcal{O}), f \in L_{loc}^2(\R_+,L^2(\mathcal{O}))$, and all $t>0$ the solution of 
 \begin{equation} 
  \left\{
    \begin{aligned}
     &  \partial_t^2 u -\gamma^2\Delta u + mu=f& \text{ in } \R_+ \times \mathcal{O},\\&
                 u=0 &  \text{ on }  \R_+ \times \partial\mathcal{O}, \\&
                 (u(0,x),\partial_t u(0,x))=(u_0,u_1)&\forall x \in\mathcal{O}. 
    \end{aligned}
  \right.
\end{equation}
where $m \geqslant 0$, satisfies with
 \begin{align*}
 E_u(t)= \frac{1}{2}\int_\mathcal{O} |\partial_t u(t,x)|^2 +m |u(t,x)|^2 + \gamma^2 |\nabla u(t,x)|^2 \ dx,
 \end{align*}
 the inequality
 \begin{equation}
 \int_t^{t+T} E_u(s) \ ds \leqslant C_T \int_t^{t+T} \int_\mathcal{O}\phi (x) |\partial_t u|^2 + |f|^2 \ dxds.
 \end{equation}
\end{lemma}

Let $\omega_{b,1} = \omega_b \cap B_{R_4} = \{x \in \Omega \cap B_{R_4}, \ b(x) > b^- > 0\}$. Since $  (\omega_b,T)$ satisfies  the  \textbf{ GCC}, $ B_{R_1^c} \subset \omega_b$  and $  R_4 >R_1$, we conclude that $ (\omega_{b,1},T)$ geometrically controls $\Omega_{R_4}$.\\
So, according to Lemma $\ref{obs}$ and using hypothesis $(\mathcal{A}_1)$,  we have
\begin{align*}
\int_t^{t+T} E_{v^i}(s)ds& \lesssim   \int_t^{t+T} \int_{\omega_{b,1}}   |\partial_t v^i |^2 \ dx ds + \int_t^{t+T} \int_{\Omega} b(x)|\partial_t u^i|^2 dx ds \\& + \int_t^{t+T}\int_{C_{R_3,R_4}}|\nabla v|^2 \ dxds+  \int_t^{t+T} \int_{\Omega_{R_4}} |v|^2 \ dx ds \\& \lesssim \int_t^{t+T} \int_\Omega b^2(x)  |\partial_t v^i |^2 \ dx ds + \int_t^{t+T} \int_{\Omega} a(x)|\partial_t u|^2 dx ds
\end{align*}
\begin{equation}
\label{energ2}
+ \int_t^{t+T}\int_{C_{R_3,R_4}}|\nabla v|^2 \ dxds+  \int_t^{t+T} \int_{\Omega_{R_4}} |v|^2 \ dx ds, \quad t >0,
\end{equation}
where 
\begin{align*}
 E_{v^i}(t) = \frac{1}{2}\int_{\Omega} |\nabla v^i(t,x)|^2 + |\partial_t v^i(t,x)|^2 + m_2|v^i(t,x)|^2 \ dx.
\end{align*}
We have also 
\begin{align*}
\int_t^{t+T} E_{u^i}(s)ds& \lesssim   \int_t^{t+T} \int_{\Omega} a(x) |\partial_t u |^2 + b^2(x)|\partial_t v^i|^2  dx ds \\&
\end{align*}
\begin{equation}
\label{energ1}
 + \int_t^{t+T}\int_{C_{R_3,R_4}}|\nabla u|^2 \ dxds+  \int_t^{t+T} \int_{\Omega_{R_4}} |u|^2 \ dx ds , \quad t >0,
\end{equation}
where
\begin{align*}
E_{u^i}(t) = \frac{1}{2}\int_{\Omega} |\nabla u^i(t,x)|^2 + |\partial_t u^i(t,x)|^2+m_1|u^i(t,x)|^2 \ dx.
\end{align*}
Adding the two estimates above and using $ \eqref{Lem6}$, we deduce that

\begin{align*}
\int_t^{t+T} &E_{u^i ,v^i}(s)ds \lesssim C_\varepsilon \int_t^{t+T}\int_{\Omega} a(x)( |\partial_t u |^2 + (1-\frac{1}{\gamma^2})^2|\partial_t^2 u|^2)  dx ds \\& + \varepsilon \int_t^{t+T}  E_{u,v}(s) \ ds + C_\varepsilon \int_t^{t+T} E^{R_3}(u,v,s) ds
\end{align*}
\begin{equation}
\label{**}
+ C_\varepsilon\int_t^{t+T}\int_{\Omega_{R_4}} | u|^2+|v|^2 dx ds  +\Big[\mathcal{K}_\gamma\Big]_t^{t+T}.
\end{equation} 
Since $\psi \equiv 1 $ for $|x|\leqslant R_3$, we get 
\begin{align*}
\int_t^{t+T} E_{R_3}(u,v,s) \ ds \leqslant \int_t^{t+T} E_{u^i ,v^i}(s)ds
\end{align*}
Combining this estimate with $\eqref{**}$, we conclude $\eqref{eningg}$.
\end{proof}
\section{Weak observability estimate}
In this section, we prove  the following proposition.
\begin{proposition}
\label{lemme3}
Let $\gamma \in  \R_+^*$ and $m_1,m_2 \in \R_+$. Let $R_1>0$ be  such that $(\mathcal{A}_2)$ is satisfied and $R_5>R_1$.   
 We assume  that the assumption $(\mathcal{A}_1)$   holds. Then for every $T> T_{\omega_b}$ and $ \alpha > 0$, there exists $C_{T,\alpha} > 0 $, such that for all $(u_0,u_1,v_0,v_1)\in (H_0^1(\Omega)\times L^2(\Omega))^2$, and all $t >0$, the solution of the system $\eqref{pb311}$ satisfies the following inequality
\begin{equation}
\label{4.9}
\int_t^{t+T}\int_{\Omega_{R_5}} |v|^2 + |u|^2 \ dx ds \leqslant C_{T,\alpha} \int_t^{t+T} \int_{\Omega} a(x) |\partial_t u |^2 \ dx ds  + \alpha \int_t^{t+T} E_{u,v}(s)\ ds.
\end{equation}
\end{proposition}
\begin{proof}[Proof of Proposition $\ref{lemme3}$]
We note that for each $(u_0,u_1,v_0,v_1) \in (H_1^0(\Omega) \times L^2(\Omega))^2$, the solution $(u,v)$  are given as the limit of smooth solutions $(u_n, v_n)(t)$ with $ (u_{n}, v_{n})(0)= (u_{n,0},v_{n,0}) \in (C_0^\infty(\Omega)) ^2 $ and $(\partial_t u_n , \partial_t v_n)(0) = (u_{n,1}, v_{n,t}) \in (C_0^\infty(\Omega)) ^2$  such that $(u_{n,0},v_{n,0})  \to (u_{0},v_{0}) \in (H_0^1(\Omega))^2$ and $ (u_{n,1}, v_{n,1}) \to (u_{1}, v_{1}) \in (L^2(\Omega))^2 $. Note that 
\begin{align*}
&\|u_n(t,.)-u(t,.)\|_{H^1} + \|\partial_t u_n(t,.)- \partial_t u(t,.) \|_{L^2}  \xrightarrow[ n \to +\infty]{\text{}} 0,\\&\|v_n(t,.)-v(t,.)\|_{H^1} + \|\partial_t v_n(t,.)- \partial_t v(t,.) \|_{L^2}  \xrightarrow[ n \to +\infty]{\text{}} 0,
\end{align*}
uniformly on the each closed interval $[0,T]$ for any $T>0$. Therefore we may assume that $(u,v)$ is smooth.

To prove the estimate $\eqref{4.9}$, we argue by contradiction.  We assume that there exist a positive sequence $(t_n)$ and a  sequence 

$$\mathcal{U}_n= (u_n,\partial_t u_n, v_n, \partial_t v_n)$$ of solution of the system $\eqref{pb311}$ with initial data $(u_{n,0}, u_{n,1}, v_{n,0}, v_{n,1})\in (H_0^1(\Omega)\times L^2(\Omega))^2$, such that 

\begin{align*}
 \int_{t_n}^{t_n+T}\int_{\Omega_{R_5}} |u_n|^2+|v_n|^2 \ dx ds  &\geqslant n  \int_{t_n}^{t_n+T}\int_\Omega a(x) |\partial_t u_n|^2 \ dx dt \\&+ \alpha \int_{t_n}^{t_n+T} E_{u_n,v_n} \ ds
\end{align*}
Set $$ \beta_n^2 = \int_{t_n}^{t_n+T}\int_{\Omega_{R_5}} |u_n|^2+|v_n|^2 \ dx ds  $$ 
and 
$$  (y_n,\partial_t y_n, z_n, \partial_t z_n)(t):= \frac{\mathcal{U}_n(t+t_n)}{\beta_n}.$$
We infer that 
\begin{align}
&\int_0^T\int_{\Omega_{R_5}} |y_n|^2+ |z_n|^2 \ dx ds=1, \\& \label{ctg} \int_0^T \int_\Omega a(x) |\partial_t y_n|^2 \ dx ds \leqslant \frac{1}{n},\\& \int_0^T E_{y_n,z_n}(s) \ ds \leqslant \frac{1}{\alpha}.
\end{align}
Therefore 
\begin{align*}
&(y_n,z_n) \rightharpoonup (y,z) \text{ in } \ L^2((0,T), H_0^1(\Omega))\cap W^{1,2}((0,T), L^2(\Omega)),
\end{align*}
with respect to the weak topology. By Rellich's lemma, we can assume  that 
\begin{align*}
(y_n,z_n) \to (y,z) \text{ in }  \ (L^2((0,T)\times \Omega_{R_5}))^2.
\end{align*}
It is easy to see that the limit $(y,z)$ satisfies the system
\begin{equation}
 \label{pbg}
  \left\{
    \begin{aligned}
     &  \partial_t^2 y -\Delta y +m_1y +b(x)\partial_t z =0&\text{ in } (0,T) \times \Omega, \\ &
  \partial_t^2 z -\gamma^2\Delta z +m_2z =0&\text{ in } (0,T) \times \Omega ,\\&
                 y=z=0&  \text{ on }  (0,T) \times \Gamma, \\&
                 a(x)\partial_t y=0& \text{ on } (0,T) \times \Omega
    \end{aligned}
  \right.
\end{equation}
and 
\begin{align}
\label{0*}
\int_0^T\int_{\Omega_{R_5}} |y|^2+|z|^2 \ dx ds =1.
\end{align}
 It is clear  that $(\partial_t y, \partial_t z)$ satisfies the following system
\begin{equation}
 \label{pbd}
  \left\{
    \begin{aligned}
     &  \partial_t^2 (\partial_t y) -\Delta(\partial_t y) +m_1\partial_t y +b(x)\partial_t(\partial_t z) =0  &\text{ in } (0,T) \times \Omega, \\ &
  \partial_t^2 (\partial_t z) -\gamma^2\Delta (\partial_tz) +m_2\partial_t z=0&\text{ in } (0,T) \times \Omega ,\\&
                 \partial_t y=\partial_t z=0 &  \text{ on }\  (0,T) \times \partial\Omega, \\&
                 a(x)\partial_t y=0& \text{ on }\  (0,T) \times \Omega.
    \end{aligned}
  \right.
\end{equation}
From the first and previous equations in (\ref{pbd}), we deduce that $b(x)\partial_t^2 z=0$ on $supp(a)$. But $supp(b)\subset supp(a)$, so $\partial_t^2 z=0$ on $supp(b)$. Setting $w= \partial_t z$, we have
\begin{equation}
\label{ed}
\left\{
    \begin{aligned}
     &\partial_t w  =0&\text{ in } (0,T) \times \omega_b, \\ &
  \partial_t^2 w -\gamma^2\Delta w +m_2w =0&\text{ in  } (0,T) \times \Omega, \\&
  w=0 &  \text{ on }\  (0,T) \times \partial\Omega,\\&
  w\in   L^2((0,T)\times\Omega).
    \end{aligned}
  \right.
\end{equation} 
Using the first and second equations in (\ref{ed}), we can see that $WF^1(w) \cap (0,T) \times \omega_b\times \R\times \R^n$ is a subset of
\begin{align*}
\{(t,x,\tau,\xi)\in (0,T) \times \Omega\times \R\times \R^n; \tau^2-\gamma^2|\xi|^2=\tau=0\}=(0,T) \times \Omega\times \{0\}\times \{0\}.
\end{align*}
where $WF^1(w)$ denotes the  $H^1$-wavefront set of $w$. Since $B_{R_1}^c \subset \omega_b$, we deduce that $w\in H_{loc}^1((0,T)\times B_{R_1}^c).$ Next, we will show that $w \in H_{loc}^1([0,T]\times R_{R_1})$. Let $\rho_0=( t_0,x_0,\tau_0,\xi_0) \in T^*([0,T]\times B_{R_1})$ and $\Gamma_0$ be the generalized bicharacteristic issued from $\rho_0 $. Set $\{\rho_1:=(0, x_1,\tau_1, \xi_1)\}= \Gamma_0 \cap \{t=0\}$  and $\{\rho_2:= (T,x_2, \tau_2, \rho_2)\} = \Gamma_0 \cap \{t=T\}$, so  we distinguish  two cases, 
 \\
\textbf{$1^{st}$ case: }  $x_1$ or $x_2 \notin B_{R_1}$. In this case $\rho_1$ or $\rho_2 \notin WF^1(w)$). Since $T> T_{\omega_b}$, then using the  propagation of regularity along the bicharacteristic flow of the operator $\partial_t^2 -\gamma^2\Delta $ (see \cite{melrose1,melrose2}),  we obtain $\rho_0 \notin WF^1(w)$.  \\ 
\textbf{$2^{nd}$ case: }  $x_1 , x_2 \in B_{R_1}$. Since $\rho_1,\rho_2\in T^*([0,T]\times B_{R_1})$ and $\omega_b$  controls geometrically  $[0,T]\times\Omega $, then $\Gamma_0$ intersects the region $[0,T]\times (\omega_b\cap \Omega_{R_1})$. But $w \in H_{loc}^1([0,T]\times (\omega_b\cap \Omega_{R_1}))$, then applying again the regularity propagation theorem, we deduce that $\rho_0 \notin WF^1(w)$. Therefore, we conclude that $w\in H_{loc}^1((0,T)\times \Omega)$. Now, set $\tilde{w}= \partial_t w $. Since $\R^n\setminus \Omega_{R_5}\subset \omega_b$, so $\tilde{w}=0$ on $\R^n\setminus \Omega_{R_5}$ and  satisfies   

\begin{equation}
\left\{
    \begin{aligned}
  & \partial_t^2 \tilde{w}  -\gamma^2\Delta \tilde{w}  +m_2\tilde{w}  =0&\text{ in  } (0,T) \times \Omega_{R_5}, \\&
  \tilde{w}=0 &  \text{ on }\  (0,T) \times \partial\Omega_{R_5},\\
  &\tilde{w}  =0&\text{ in } (0,T) \times (\omega_b \cap \Omega_{R_5}) , \\&
 \tilde{w} \in   L^2((0,T)\times\Omega_{R_5})
    \end{aligned}
  \right.
\end{equation} 
Since $\omega_b \cap \Omega_{R_5}$ controls geometrically $\Omega_{R_5}$, then using the classical unique continuation result (see \cite{e8,Burq-Gerard} ), we infer that $\tilde{w} \equiv 0 $ on $(0,T)\times\Omega_{R_5}$. Therefore, the function $z$ satisfies 
\begin{equation}
\label{gh}
\left\{
    \begin{aligned}     
  & -\gamma^2\Delta z + m_2z =0&\text{ in } (0,T) \times \Omega ,\\ &z =0&\text{ in } (0,T) \times \partial\Omega.
    \end{aligned}
  \right.
\end{equation}  
This implies that $z=0$ on $ (0,T) \times\Omega$.  Now, from $\eqref{pbg}$ we obtain 
\begin{equation}
 \label{df}
 \left\{
    \begin{aligned}     
  & \partial_t^2y-\Delta y +m_1y =0&\text{ in  } (0,T)\times \Omega ,\\ & a(x) \partial_t y  =0&\text{ in  } (0,T) \times \Omega, \\& y = 0  &\text{ on  } (0,T) \times \partial\Omega,\\&
  y\in H^1((0,T)\times \Omega)
    \end{aligned}
  \right.
\end{equation} 
Arguing as for $z$, we can prove that $y=0$.  This is in contradiction with  $\eqref{0*}$.
\\ 

\end{proof}

 \section{Proof of Theorem $\ref{theoreme1}$}
Let $R_2 >R_1$.  According to $\eqref{enex1g}$ for $t=nT$, $n\in \N^*$, we have 
\begin{align*}
\int_0^{nT} E^{R_2}(u,v,s)ds \lesssim  C_\varepsilon \Bigg( E_{u,v}(0)+ (1-\frac{1}{\gamma^2})^2 E_{\partial_t u, \partial_t v}(0) +  \int_0^{nT}\int_{\Omega_{R_2}} | u|^2 
\end{align*}
\begin{equation}
\label{4.6}
+|v|^2 dx ds \Bigg) + \varepsilon \int_0^{nT}  E_{u,v}(s) \ ds  + \|(u,v)(0)\|_{L^2}^2.
\end{equation}
Next, using $\eqref{eningg}$ with $R_3= 2R_2$ and $R_4=3R_2$, we get 
\begin{align*}
\int_{kT}^{(k+1)T}& E_{2R_2}(u,v,s) ds \lesssim  C_\varepsilon\int_{kT}^{(k+1)T}\int_{\Omega} a(x)( |\partial_t u |^2 + (1-\frac{1}{\gamma^2})^2|\partial_t^2 u|^2)  dx ds \\& + \varepsilon \int_{kT}^{(k+1)T} E_{u,v}(s) \ ds +  C_\varepsilon \int_{kT}^{(k+1)T} E^{2R_2}(u,v,s)ds
\end{align*}
\begin{equation}
\label{3.100}
+ C_\varepsilon\int_{kT}^{(k+1)T}\int_{\Omega_{3R_2}} | u|^2+|v|^2 dx ds  - \Big[\mathcal{K}_\gamma\Big]_{kT}^{(k+1)T} , \forall \ k \in \N.
\end{equation}
Thus 
\begin{align*}
&\sum_{k=0}^{n-1}\int_{kT}^{(k+1)T} E_{2R_2}(u,v,s)ds \lesssim \sum_{k=0}^{n-1} \Bigg( C_\varepsilon\int_{kT}^{(k+1)T}\int_{\Omega} a(x)( |\partial_t u |^2 \\&+ (1-\frac{1}{\gamma^2})^2|\partial_t^2 u|^2)  dx ds  + \varepsilon \int_{kT}^{(k+1)T} E_{u,v}(s) \ ds  - \Big[\mathcal{K}_\gamma\Big]_{kT}^{(k+1)T}
\end{align*}
\begin{equation}
+  C_\varepsilon \Big(\int_{kT}^{(k+1)T}E^{2R_2}(u,v,s)ds +\int_{kT}^{(k+1)T}\int_{\Omega_{3R_2}} | u|^2+|v|^2 dx ds \Big)\Bigg) , \forall \ k \in \N.
\end{equation}
This gives
\begin{align*}
\int_0^{nT} E_{2R_2}&(u,v,s) \ ds \lesssim  C_\varepsilon\int_0^{nT}\int_{\Omega} a(x)( |\partial_t u |^2 + (1-\frac{1}{\gamma^2})^2|\partial_t^2 u|^2)  dx ds \\& + \varepsilon \int_0^{nT}  E_{u,v}(s) \ ds +  C_\varepsilon \int_0^{nT} E^{2R_2}(u,v,s) \ ds
\end{align*}
\begin{equation}
+  C_\varepsilon\int_0^{nT}\int_{\Omega_{3R_2}} | u|^2+|v|^2 dx ds  - \Big[\mathcal{K}_\gamma\Big]_0^{nT} , \forall \ n \in \N^*.
\end{equation}
From the following estimate
\begin{align*}
&\Big|\mathcal{K}_\gamma(s)\Big| \lesssim E_{u,v}(0), \forall s\geqslant 0, 
\end{align*}
and using $\eqref{decroissance}$ and $\eqref{4.6}$, we deduce that 
\begin{align*}
\int_0^{nT}E_{2R_2}(u,v,s) \ ds  \lesssim  C_\varepsilon( E_{u,v}(0) + (1-\frac{1}{\gamma^2})^2 E_{\partial_t u, \partial_t v}(0)) \
\end{align*}
\begin{equation}
\label{4.5}
+ \varepsilon \int_0^{nT}  E_{u,v}(s) \ ds+  C_\varepsilon \int_0^{nT}\int_{\Omega_{3R_2}} | u|^2+|v|^2 dx ds  , \forall \ n \in \N^*.
\end{equation}
So, combining $\eqref{4.5}$ and $\eqref{4.6}$, we conclude for small enough $\varepsilon$ the following estimate 
\begin{align*}
\int_0^{nT} E_{u,v}(s)ds \lesssim  C_\varepsilon( E_{u,v}(0)+ (1-\frac{1}{\gamma^2})^2 E_{\partial_t u, \partial_t v}(0) )  
\end{align*}
\begin{equation}
\label{enn} 
\|(u,v)(0)\|_{L^2}^2 + C_\varepsilon \int_0^{nT}\int_{\Omega_{3R_2}} (|v|^2 + |u|^2) \ dx ds.
\end{equation}
Next, From $\eqref{4.9}$ with $R_5=3R_2$ we have 
\begin{align*}
\sum_{k=0}^{n-1}\int_{kT}^{(k+1)T}\int_{\Omega_{3R_2}} |v|^2 + |u|^2 \ dx ds &\lesssim \sum_{k=0}^{n-1}\Big(\int_{kT}^{(k+1)T} \int_{\Omega} a(x) |\partial_t u |^2 \ dx ds \\& + \alpha \int_{kT}^{(k+1)T} E_{u,v}(s)\ ds \Big).
\end{align*}
Thus
\begin{equation}
\label{4.10}
\int_0^{nT}\int_{\Omega_{3R_2}} |v|^2 + |u|^2 \ dx ds \lesssim E_{u,v}(0)  + \alpha \int_0^{nT} E_{u,v}(s)\ ds.
\end{equation}
Finally, using $\eqref{4.10}$ for $\alpha$ small enough in $\eqref{enn}$, we find

\begin{equation}
\label{4.11}
\int_0^{nT} E_{u,v}(s)ds \lesssim  C_\varepsilon ( E_{u,v}(0) + (1-\frac{1}{\gamma^2})^2 E_{\partial_t u ,\partial_t v}(0))+\|(u,v)(0)\|_{L^2(\Omega)}^2,
\end{equation}
Therefore
\begin{align*}
\int_0^{+\infty} E_{u,v}(s)ds \lesssim  E_{u,v}(0) +  (1-\frac{1}{\gamma^2})^2 E_{\partial_t u ,\partial_t v}(0)+ \|(u,v)(0)\|_{L^2(\Omega)}^2.
\end{align*}
As the energy is decreasing then
\begin{align*}
(1+t)E_{u,v}(t) &\leqslant \int_0^{+\infty} E_{u,v}(s)ds  +E_{u,v}(0) \\&\lesssim E_{u,v}(0)+ (1-\frac{1}{\gamma^2})^2 E_{\partial_t u ,\partial_t v}(0)
\end{align*}
\begin{equation}
\label{4.12}
+  \|(u,v)(0)\|_{L^2(\Omega)}^2, \text{ for all }  t> 0.
\end{equation}
On the other hand, using $\eqref{enex1g}$, $\eqref{4.10}$ and $\eqref{4.11}$, we deduce that
\begin{equation}
\label{rtr}
\int_\Omega \varphi (|u(t)|^2 +  |v(t)|^2 ) \ dx \lesssim  E_{u,v}(0) + (1-\frac{1}{\gamma^2})^2 E_{\partial_t u ,\partial_t v}(0)+\|(u,v)(0)\|_{L^2(\Omega)}^2.
\end{equation}
Since  $\varphi \equiv 1$ for $|x| \geqslant R_2$,
 \begin{align}
 \label{4012}
 \int_\Omega \varphi ( |u(t)|^2+|v(t)|^2)  \ dx \geqslant  \int_{\Omega_{R_2}^c}  |u(t)|^2+|v(t)|^2 \ dx ,
 \end{align}
 therefore
 \begin{equation}
 \label{4.13}
  \int_{\Omega_{R_2}^c}  |u(t)|^2+|v(t)|^2 \ dx \lesssim  E_{u,v}(0) + (1-\frac{1}{\gamma^2})^2 E_{\partial_t u ,\partial_t v}(0)+\|(u,v)(0)\|_{L^2(\Omega)}^2.
 \end{equation}
Poincare's inequality and the fact that the energie of $(u,v)$ is decreasing gives 
\begin{align}
\label{4.14}
\int_{\Omega_{3R_2}} |u(t)|^2+|v(t)|^2  \ dx \leqslant C_{\Omega} \int_{\Omega_{3R_2}} |\nabla u(t)|^2 +|\nabla v(t)|^2  \  dx \lesssim E_{u,v}(0)
\end{align}
for all $ t > 0$.\\
Adding $\eqref{4.14}$ and $ \eqref{4.13}$, we infer that
\begin{align}
\label{4.16}
\int_{\Omega}  |u(t)|^2+|v(t)|^2 \ dx \lesssim  E_{u,v}(0) + (1-\frac{1}{\gamma^2})^2 E_{\partial_t u ,\partial_t v}(0)+\|(u,v)(0)\|_{L^2(\Omega)}^2,
\end{align} 
for all $ t > 0$.
\begin{proof}[Proof of  Corollary $\ref{corol}$]
From $\eqref{4.12}$, we deduce if $\gamma =1$
\begin{align*}
E_{u,v}(t) \leqslant\frac{C}{t}E_{u,v}(0), \qquad \text{ for all }  t> 0,
\end{align*}
we choose $t$ such that $\frac{C}{t} <1$ and using the semi-group proprety, we conclude that the estimate $\eqref{theorem 2}$. 
\\ and  if $\gamma \neq 1$,  
\begin{align*}
E_{u,v}(t) \leqslant\frac{C}{t}(E_{u,v}(0)+ (1-\frac{1}{\gamma^2})^2 E_{\partial_t u, \partial_t v}(0)), \qquad \text{ for all }  t> 0,
\end{align*}
 according to  [Theoreme $2.1$, \ref{e1}] we infer that  $\eqref{theorem 22}$.

\end{proof}

\bibliographystyle{abbrv}	


\end{document}